\documentclass{rspublic}

\usepackage{amsfonts}
\usepackage{graphicx}
\usepackage[latin1]{inputenc}
\usepackage{amsmath,amssymb,amsthm}
\usepackage{calc}

\newtheorem{thm}{Theorem}

\newcommand{\al}{\alpha}
\newcommand{\sigmab}{\bar{\sigma}}

\begin{document}

\title[Divergence of Euler's method]{Strong 
and weak divergence in finite time of 
Euler's method for stochastic differential
equations with non-globally
Lipschitz continuous coefficients}

\author[M. Hutzenthaler, A.\ Jentzen 
\& P.\ E.\ Kloeden]
{Martin Hutzenthaler$^1$, Arnulf Jentzen$^2$ and Peter E. Kloeden$^3$}

\affiliation{
$^1$LMU Biozentrum,
Department Biologie II,
University of Munich (LMU), 
D-82152~Planegg-Martinsried, Germany,
e-mail: hutzenthaler$\,$(at)$\,$bio.lmu.de
\\
$^2$Program in Applied 
and Computational Mathematics, 
Princeton University,
Princeton, NJ 08544-1000, 
USA, e-mail: ajentzen$\,$(at)$\,$math.princeton.edu
\\
$^3$Institute for Mathematics, 
Goethe University Frankfurt am Main,
D-60054 Frankfurt am Main, Germany, 
e-mail: kloeden$\,$(at)$\,$math.uni-frankfurt.de
} 
\label{firstpage}
%\footnote[0]{{\it AMS\/\ subject
%classifications {\rm 65C30}}. }\footnote[0]{{\it

\maketitle

\begin{abstract}{Euler scheme,
Euler-Maruyama,
stochastic differential equations,
weak approximation,
weak divergence,
strong approximation,
strong divergence, 
non-globally Lipschitz continuous}
The stochastic Euler scheme is 
known to converge to the exact solution of
a stochastic differential equation
with globally Lipschitz continuous
drift and diffusion 
coefficient.
Recent results extend this convergence to 
coefficients which grow at most linearly.
For superlinearly growing coefficients 
finite-time convergence 
in the strong mean square 
sense remained an open question according
to [Higham, Mao \& 
Stuart (2002); Strong convergence 
of Euler-type methods for 
nonlinear stochastic 
differential equations, 
{\it SIAM J.\ Numer.\ Anal.} 
{\bf 40}, no.\ 3, 1041-1063].
In this article we answer this question
to the negative and prove for a large class of stochastic differential equations with 
non-globally Lipschitz continuous 
coefficients that
Euler's approximation converges neither
in the strong mean square sense 
nor in the numerically weak sense
to the exact solution  
at a finite time point. 
Even worse,
the difference
of the exact solution
and of the numerical
approximation 
at a finite time point
diverges 
to infinity
in the strong mean square sense 
and in the numerically weak sense.
\end{abstract}

%\def\thefootnote{\fnsymbol{footnote}}
%\def\@makefnmark{\hbox to\z@{$\m@th^{\@thefnmark}$\hss}}
%\footnotesize\rm\noindent
%\footnote[0]{{\it AMS\/\ subject
%classifications {\rm 65C30}}. }
%\normalsize\rm

\section{Introduction}%
\label{sec:introduction}
An important
numerical scheme for simulating stochastic differential
equations (SDEs) is Euler's method 
(see, e.g., Kloeden \&
Platen~1992,
Milstein~1995 
and
Higham~2001).
If the coefficients of an SDE are globally Lipschitz continuous,
then standard results 
(see, for instance,
Chapter~10 and Chapter~14 
in Kloeden \&
Platen~1992)
show convergence of the Euler approximation 
in the strong and numerically weak 
sense to the exact solution of the SDE.
It remained an open question whether the Euler approximation
also converges in the strong or 
numerically weak sense at a finite time point
if the coefficients 
of the SDE are not 
globally Lipschitz continuous,
see, e.g., 
Section~1 in Higham,
Mao \&
Stuart~(2002)
for a detailed
description of this open problem.
In this paper we answer this question to the negative.
More precisely, we prove for a large class of SDEs with 
superlinearly growing coefficient functions that both
the distance in the strong $L^p$-sense
and
the distance between 
the $p$-th 
absolute moments
of 
the Euler approximation
and 
of the exact solution of the SDE
diverge
to infinity for all $ p \in [1,\infty)$.
Thus the Euler scheme 
does not produce
an approximation
in the strong or numerically weak sense
of the exact solution of such an SDE.
% at finite time points. 

For clarity of exposition we concentrate in this section on the
following prominent example
% (also considered in \cite{MSH02,MT05,Ta02}).
Let $(X_t)_{t\geq0}$ be the unique
solution process
of the
one-dimensional SDE
\begin{equation}\label{exsde2}
d X_t =
- X_t^3\,dt + dW_t,
\qquad
X_0 = x_0
\end{equation}
for $ t \geq 0 $
where $(W_t)_{t\geq0}$ 
is a one-dimensional
standard 
Brownian motion and 
where $ x_0 \in \mathbb{R} $
is a constant.
The Euler approximation 
$
  (Y_n^N)_{
    n \in 
    \left\{ 0,1,\ldots,N
    \right\}
  }
$ 
of 
$\left( X_t \right)_{t \in [0,T]}$, 
$T\in(0,\infty)$ fixed,
is defined recursively through
$ Y^N_0 := x_0 $ and
\begin{equation}\label{exeuler2}
Y^N_{n+1}
:=
Y^N_n
-
\frac{T}{N}
\left( Y^N_n \right)^3
+
\left(
W_{ \frac{(n+1) T}{N} }
-
W_{ \frac{ n T }{N} }
\right)
\end{equation}
for all $n \in \left\{ 0,1,2, \ldots \right\} $
and all $ N \in \mathbb{N}:=\{1,2,\ldots\}$.
%The main observation of this article is that the absolute
%moments of the Euler approximation~\eqref{exeuler2} at a finite time point
%diverge to infinity, that is,
%\begin{equation}     \label{eq:unbounded.moments}
%  \lim_{N\to\infty}
%  \mathbb{E}\Big[\big| Y_N^N \big|^p\Big]
%  =\infty
%\end{equation}
%for all $p\in[1,\infty)$.

Two results motivated us to try to prove convergence of the
stochastic Euler approximation $ Y^N_N $
to the exact solution $ X_T $ of the
SDE~\eqref{exsde2} in the strong mean
square sense 
as the number of time steps 
$ N $ goes to infinity.
Gy\"ongy~(1998) established pathwise convergence for SDEs with locally 
Lipschitz continuous coefficients. 
More precisely, Theorem~1 in
Gy\"ongy~(1998) implies
\begin{equation}
\label{eq:pathwise}
  \lim_{ 
    N \rightarrow \infty 
  }
  \left|
    X_T
    - Y^N_N
  \right|
  = 0
\end{equation}
$\mathbb{P}$-a.s.\ pathwise 
convergence of the stochastic Euler
approximation~\eqref{exeuler2}
to the exact solution
of the SDE~\eqref{exsde2}.
This implies convergence of 
expectations of continuous and bounded 
functionals of the difference
between the Euler approximation and the exact solution.
Of course, the squared difference needed for mean square convergence
is not a bounded 
%functional 
but an unbounded functional.
Another motivation was an instructive 
conditional result of
Higham,
Mao \& Stuart~(2002).
They assume in their Theorem 2.2 
local Lipschitz continuity of the coefficients
of the SDE and boundedness of the $p$-th moment
of the Euler approximation and of the exact solution 
in the sense that
\begin{equation}
\label{eq:bounded_mom}
 \sup_{ N \in \mathbb{N} }
 \mathbb{E}\left[
   \sup_{ 
     n \in \{ 0, 1, 
     \ldots, N \} 
   }
   \left| Y_n^N \right|^{ p }
 \right]
 +\mathbb{E}\left[
    \sup_{ t \in [0,T] }
    \left| X_t \right|^p
  \right]
 <
 \infty 
\end{equation}
for one arbitrary $ p \in (2,\infty) $.
Under these assumptions they establish
in Theorem~2.2 in 
Higham,
Mao \& Stuart~(2002)
strong mean square convergence
of Euler's method
to the exact solution
of the SDE.
Under
assumption~\eqref{eq:bounded_mom}
they, in particular, 
establish
strong mean square 
convergence
\begin{equation}
 \lim_{ N \rightarrow \infty }
 \mathbb{E}\Big[\left|
   X_T - Y_N^N
 \right|^2\Big]
 = 
 0
\end{equation}
of the Euler approximation~\eqref{exeuler2} 
to the exact solution of the SDE~\eqref{exsde2}.
Boundedness of moments~\eqref{eq:bounded_mom}
of the Euler approximation, however,
remained an open question.
Higham,
Mao \& Stuart~(2002)
say on page 1060 that
in ``general, it is not clear when such moment bounds can
be expected to hold for explicit
methods with $f,g\in C^1$''
($f$ and $g$ denote 
in Higham,
Mao \& Stuart~(2002) the drift
function and the diffusion function respectively).

In this article we answer
Higham, Mao \&
Stuart's question
in the
case of the explicit 
Euler method
and superlinearly growing
coefficients of the SDE
to the negative 
(see Theorem~\ref{thm:euler}
below for details)
and establish strong 
$ L^p $-divergence
\begin{equation}  \label{eq:exstrong}
\lim_{N \rightarrow \infty}
\mathbb{E}\Big[\left|
X_T - Y^N_N
\right|^p\Big]
= 
\infty
\end{equation}
of the stochastic
Euler appoximation~\eqref{exeuler2}
in finite time $T\in(0,\infty)$
for all $ p \in [1,\infty)$
(see~\eqref{eq:euler_divergence} for details).
In addition numerically weak convergence (see, e.g., 
Section 9.4 in 
Kloeden \&
Platen~1992)
(not to be confused with stochastic weak convergence)
fails to hold, even more,
\begin{equation}  \label{eq:exweak}
\lim_{N \rightarrow \infty}
\biggl|
\mathbb{E}\Big[\left|X_T\right|^p\Big]
-\mathbb{E}\Big[\left|Y_N^N\right|^p\Big]
\biggr|
=\infty
\end{equation}
for all $ p \in [1,\infty)$,
see~\eqref{eq:euler_divergence}.
In particular,
Theorem~\ref{thm:euler} implies
that the absolute
moments of the Euler approximation~\eqref{exeuler2} at a finite time point
diverge to infinity, that is,
\begin{equation}     
\label{eq:unbounded.moments}
 \lim_{N\to\infty}
 \mathbb{E}\Big[\big| Y_N^N \big|^p\Big]
 =\infty
\end{equation}
for all $p\in[1,\infty)$.
Thus the moment bound assumption in
Theorem~2.2 
in Higham, Mao \&
Stuart~(2002)
(see also inequality~\eqref{eq:bounded_mom} here) 
is not satisfied
for the SDE~\eqref{exsde2}.

Note that the strong divergence~\eqref{eq:exstrong}
and the weak divergence~\eqref{eq:exweak}
of the Euler approximation is not
a special property of equation \eqref{exsde2}.
We establish this divergence for a large 
class of SDEs with superlinearly
growing coefficients in
Section \ref{secexamples}.
Moreover, 
our estimates are easily adapted to prove divergence of
other numerical schemes such as the Milstein scheme. For simplicity
we restrict ourselves to the Euler scheme.
Presence of noise, however, is essential. 
In the deterministic case,
the Euler scheme does converge
and both~\eqref{eq:exstrong} and~\eqref{eq:exweak} fail to hold.

Next we relate the divergence result~\eqref{eq:unbounded.moments}
regarding finite time intervals to a divergence result
of Mattingly, Stuart \& Higham~(2002)
regarding infinite time intervals
(see also 
Theorem~3.2 in 
Roberts \& 
Tweedie 1996 and
Lemma~4.1 in 
Higham, Mao \&
Stuart 2003).
Their 
Lemma 6.3 
shows
for the SDE~\eqref{exsde2}
that the second moment of the
Euler approximation diverges in infinite time
for any fixed discretization step size, that is,
\begin{equation}  \label{eq:unbounded.infinite.time}
  \lim_{n\to\infty}
  \mathbb{E}\Big[\big| Y_n^N \big|^2\Big]
  =\infty
\end{equation}
for every fixed $N\in\mathbb{N}$.
Note that this divergence 
result is rather different
to our divergence result~\eqref{eq:unbounded.moments}.
First the time discretisation step size does not converge
to zero in~\eqref{eq:unbounded.infinite.time}.
%Over an infinite time interval this might lead to a different
%behavior of the Euler approximation compared to the exact
%solution.
%the absolute moments of the Euler approximation
%on a finite time interval
%diverge although the paths of the Euler approximation converge
%almost surely.
Secondly
the Euler approximation diverges
even pathwise with positive probability on an infinite time interval.
More precisely,
Lemma 6.3 in
Mattingly, Stuart \& Higham~(2002)
proves
\begin{equation}
  \mathbb{P}\Big[
    \big| Y_n^N \big| 
    \geq 2^n \sqrt{N}
    \;\;\forall \, n \in \mathbb{N}
  \Big]
  >0
\end{equation}
for every fixed $N\in\mathbb{N}$.
The divergence~\eqref{eq:unbounded.infinite.time}
is then an immediate consequence
of this pathwise divergence result.
On a finite time interval the Euler approximation does converge pathwise to the exact solution 
according to 
Gy\"ongy (1998)
(see \eqref{eq:pathwise} here).
So here the divergence~\eqref{eq:unbounded.moments} 
of the moments is not a consequence
of the pathwise behavior.

In the next step we 
compare the divergence
results~\eqref{eq:exstrong}-\eqref{eq:unbounded.moments}
in this article
with a few further related 
results in the literature.
It is a classical result 
(see, e.g., Kloeden
\& Platen~1992
and 
Milstein~1995)
that the Euler approximation
converges in case of globally 
Lipschitz continuous 
drift and diffusion 
coefficient functions.
In contrast to our results on
superlinearly growing coefficients,
Yan (2002) proves numerically weak convergence of the Euler scheme 
if both drift and diffusion function have at most linear growth.
Zhang (2006) 
and
Berkaoui, Bossy 
and Diop~(2008)
prove strong convergence
for a class of drift and
diffusion functions with a singularity.
Yuan \& Mao (2008)
obtain the rate of strong $L^2$-convergence of the stochastic Euler scheme
for locally Lipschitz continuous 
coefficients if these grow at 
most linearly.
An explicit bound for the 
strong $L^1$-error is established in
Higham \& Mao (2005)
for the mean-reverting square-root process 
(a.k.a.\ CIR process)
which has linearly bounded coefficients.
Bernard \& Fleury (2001) establish convergence in probability under weak
hypotheses like local Lipschitz continuity.
Pathwise convergence results can be found in Gy\"ongy (1998),
Fleury (2005),
Jentzen \& Kloeden (2009) or in
Jentzen, Kloeden
\&
Neuenkirch~(2009).
A number of authors obtain convergence for
modified Euler schemes.
Milstein \& Tretyakov (2005) consider a modified Euler scheme for
non-globally Lipschitz continuous 
coefficients. They obtain 
numerically
weak convergence by 
discarding
trajectories of the 
approximation 
which leave a sufficiently large sphere.
Lamba, Mattingly \& 
Stuart~(2007) prove strong convergence of an Euler scheme
with an adaptive time-stepping algorithm for locally Lipschitz continuous
coefficients.
A completely different adaptive time-stepping algorithm with focus on long time
approximation is proposed by Lemaire (2007).
Finally, note that
in contrast to the
explicit Euler 
method~\eqref{exeuler2},
the implict Euler method
converges in the root mean
square sense
according to
Higham, Mao \& Stuart~(2002)
(see also
Hu~1996,
Talay~2002,
Szpruch \&
Mao~2010
and the references therein
for more convergence
results on implicit numerical
methods for SDEs).

The rest of this article is
organized as follows.
Our main result
(Theorem~\ref{thm:euler})
and several examples 
for the divergence of the
Euler scheme are provided
in Section~\ref{secgeneral}.
Simulations in Section~\ref{sec:simulations}
illustrate this divergence.
%While a conclusion of this 
%article is presented in Section~\ref{sec:conclusion},
The proof of Theorem~\ref{thm:euler} 
is postponed
to the final section.

\section{Main result and examples}
\label{secgeneral}
\label{secexamples}

Throughout this section
assume that the following setting
is fulfilled.
Fix $T\in(0,\infty)$ and let 
$ \left( \Omega, \mathcal{F}, \mathbb{P}
\right)$ be a probability space
with normal filtration
$ \left( \mathcal{F}_t 
\right)_{ t \in [0,T] } $.
Additionally, let
$ W\colon [0,T]
\times \Omega \rightarrow \mathbb{R}$
be a one-dimensional standard 
$\left( \mathcal{F}_t 
\right)_{t \in [0,T]}$-Brownian
motion 
and let 
$ \xi \colon \Omega \rightarrow \mathbb{R}$
be an $ \mathcal{F}_0 $/$\mathcal{B}(\mathbb{R})$-measurable
mapping.
Moreover, let $ \mu, \sigma \colon \mathbb{R}
\rightarrow \mathbb{R} $ be two
$ \mathcal{B}(\mathbb{R}) $/$\mathcal{B}(\mathbb{R})$-measurable
functions such that
the SDE
\begin{equation}
\label{generalsde}
dX_t = \mu(X_t)\,dt
+ \sigma(X_t)\,dW_t,
\qquad
X_0 = \xi
\end{equation}
for $ t \in [0,T] $
has a
solution.
More precisely, we assume
the existence of a 
predictable stochastic 
process
$ X \colon [0,T] \times
\Omega \rightarrow \mathbb{R} $
satisfying
\begin{equation}\label{exactf2}
\mathbb{P}\Bigg[
 \, \int^T_0 \left| \mu(X_s) \right| 
 + \left| \sigma(X_s) \right|^2 
 ds
 < \infty \,
\Bigg]
= 1
\end{equation}
and
\begin{equation}\label{exactf}
\mathbb{P}\Bigg[ \,
 X_t = \xi + \int^t_0 \mu(X_s)\,ds
 +
 \int^t_0 \sigma(X_s) \, dW_s \,
\Bigg] = 1
\end{equation}
for all $ t \in [0,T]$.
The drift function $\mu(\cdot)$ is the infinitesimal mean 
of the process $X$
and the diffusion function $\sigma(\cdot)$ is
the infinitesimal standard deviation of the process $X$.
The Euler approximation
for the SDE \eqref{generalsde} --
denoted by
$\mathcal{F}$/$\mathcal{B}(\mathbb{R})$-measurable mappings 
$ Y^N_n \colon \Omega \rightarrow
\mathbb{R} $, 
$ n \in \left\{ 0, 1, \dots, N \right\} $,
$ N \in \mathbb{N}$, --
is given by
$ Y^N_0(\omega) := \xi(\omega)
$
and
\begin{equation}\label{corexpeuler}
Y^N_{n+1}(\omega) := 
Y^N_{n}(\omega)
+ \frac{T}{N}
\cdot
\mu\left( 
  Y^N_{n}(\omega)
\right)
+
\sigma\!\left( Y^N_n(\omega) \right)
\cdot
\left( W_{ \frac{(n+1) T}{N} }(\omega)
- W_{ \frac{n T}{N} }(\omega)
\right)
\end{equation}
for all 
$ n \in \left\{ 0,1,\dots,N-1 \right\} $,
$ N \in \mathbb{N}$
and all $ \omega \in \Omega $.
Now we formalize the main result
of this article which asserts
the divergence of Euler's method
for the SDE \eqref{generalsde}
if at least one
coefficient grows superlinearly.
\begin{thm}\label{thm:euler}
Assume that the setting
above is fulfilled with
$
 \mathbb{P}\big[ \,
   \sigma(\xi) \neq 0 \,
 \big]
 > 0
$
and let
$ C \geq 1 $,
$ \beta > \alpha > 1 $ be constants
such that
\begin{equation}   \label{eq:maxmin_condition}
 \max\!\big(\left| \mu(x) \right|,
 \left| \sigma(x) \right|
 \big) \geq 
 \frac{
   \left| x \right|^{\beta}
 }{C} 
 \quad
  \text{and}
 \quad
 \min\!\big(\left| \mu(x) \right|,
 \left| \sigma(x) \right|
 \big) \leq 
 C | x |^{\alpha}
\end{equation}
for all $ |x| \geq C $.
Then there exists a constant
$ c \in (1,\infty) $ and a sequence of 
nonempty events
$ \Omega_N \in \mathcal{F} $,
$ N \in \mathbb{N} $, 
with 
$
  \mathbb{P}\big[ 
    \Omega_N
  \big]
  \geq 
  c^{ 
    \left( 
      - N^c
    \right) 
  }
$
and 
$
  \left| 
    Y^N_N( \omega)
  \right|
  \geq
  2^{ 
    \left( 
      \alpha^{ 
        \left( 
          N - 1 
        \right) 
      }
    \right) 
  }
$
for all $ \omega \in \Omega_N $
and all $ N \in \mathbb{N} $.
Moreover, if the exact
solution $ X : [0,T] \times \Omega
\rightarrow \mathbb{R} $ 
of the 
SDE~\eqref{generalsde}
satisfies 
$ \mathbb{E}\big[ | X_T |^p \big]
< \infty $ for one
$p \in [1,\infty)$, then
\begin{equation}   
\label{eq:euler_divergence}
  \lim_{N \rightarrow \infty}
  \mathbb{E}\Big[
    \left|
      X_T - Y^N_N
    \right|^p
  \Big] 
  = \infty
  \quad
  \text{and}
  \quad
  \lim_{N \rightarrow \infty}
  \bigg| 
    \mathbb{E}\Big[
      \left|
        X_T 
      \right|^p
    \Big]
    - 
    \mathbb{E}\Big[
      \left| Y^N_N
      \right|^p 
    \Big]
  \bigg| 
  = \infty .
\end{equation}
% where 
% $ Y^N_k $, 
% $k \in \left\{ 0,1,\dots,N
% \right\} $,
% $N \in \mathbb{N}$, is defined
% through \eqref{corexpeuler}.
\end{thm}
\noindent
Roughly speaking, 
the theorem asserts that in 
the presence of 
noise there exists a sequence of events 
of at least exponentially small probability
on which 
%the deterministic 
%dynamics lets 
the Euler approximations 
grow at least double-exponentially 
fast.
%Moreover, the probability of these 
%events is at least exponentially
%small.
Consequently as being 
{\bf double-exponentially large} 
over-compensates that
the events have at least
{\bf exponentially small} probability, 
the $L^1$-norm of the
Euler approximations $Y_N^N$ 
are unbounded in $N\in\mathbb{N}$.
The proof of the double exponential growth
of the Euler approximations 
is based on elementary calculations.
For details the reader is referred to
Section~\ref{sec:proofs} where the proof
of Theorem~\ref{thm:euler} can be
found.

Condition~\eqref{eq:maxmin_condition} should be read as follows.
Either the drift function grows in a
higher polynomial order than linearly and 
the diffusion function function grows slower
than that or the diffusion function
grows in a higher polynomial order than
linearly and the drift function
grows slower than that.
More formally it suffices to show 
that either
\begin{equation}\label{eq:fdom}
 \left| \mu(x) \right| \geq \frac{ \left| x \right|^{\beta} }{C} ,
 \qquad
 \left| \sigma(x) \right| \leq C \left| x \right|^{\alpha}
\end{equation}
for all $|x|\geq C$
or
\begin{equation}\label{eq:gdom}
 \left| \sigma(x) \right| \geq \frac{ \left| x \right|^{\beta} }{C} ,
 \qquad
 \left| \mu(x) \right| \leq C \left| x \right|^{\alpha}
\end{equation}
for all $|x|\geq C$
and
some constants 
$\beta>1$, $\beta>\alpha\geq0$, $C>0$.
Note that our estimates 
need $\beta>1$. 
A drift
function of the form 
$\mu(x)=x \log|x|$ 
for all $ x \in \mathbb{R} $ 
is too small for our estimates.
The assumption that the diffusion function does not vanish on the starting point
ensures the presence of noise 
in the first time step.

In the remainder of this section we apply Theorem~\ref{thm:euler}
to a selection of examples.
%%%Readers interested in simulations are refered to the example of the
%%%stochastic Ginzburg-Landau equation. We choose this equation for simulations
%%%because there exists an explicit solution and because it resembles our
%%%introductory example.
Note that
the coefficients in the following examples 
satisfy an appropriate one-sided 
linear growth condition.
Therefore, the $p$-th absolute moment
of the exact solution is finite 
in each example for every $ p \in [1,\infty) $
according to Theorem 2.4.1 in
Mao~(1997). 
For all examples we check that assumption~\eqref{eq:fdom} 
is satisfied. Thus both
the distance in the strong $L^p$-sense
and
the distance between 
the $p$-th 
absolute moments
of 
the Euler approximation
and 
of the exact solution 
diverges
to infinity for every $p \in [1,\infty)$
in each of the
following examples.
\subsection{The introductory
example}
The example in Section~\ref{sec:introduction} is
\begin{equation}
dX_t = -X_t^3 \, dt 
 + dW_t,
\qquad
X_0 = x_0 \in\mathbb{R}
\end{equation}
for $ t \in [0,T] $. The dominating 
coefficient is the drift function with dominating exponent $\beta=3$.
The diffusion function has exponent zero and we may choose $\alpha=0$.
The constant $C$ in condition~\eqref{eq:fdom} can be chosen as $C=1$.

\subsection{Stochastic 
Ginzburg-Landau equation}
The Ginzburg-Landau equation is from the theory of superconductivity.
It has been introduced by Ginzburg \& Landau (1950)
to describe a phase transition.
Its stochastic version with multiplicative noise can be written as
\begin{equation}     \label{eq:ginzburg-landau}
 d X_t = 
 \left(  
   \left( \eta+\frac{1}{2}\sigmab^2
   \right) X_t
   -
   \lambda X_t^3 
 \right) dt 
+ \sigmab X_t \, dW_t,
\qquad
X_0 = x_0\in(0,\infty)
\end{equation}
for $ t \in [0,T] $
where $\eta \geq 0$, 
$\lambda,\sigmab>0$.
Its solution is known explicitly (see, e.g., Section 4.4 in
Kloeden \& 
Platen~1992)
\begin{equation} \label{eq:ginzburg-landau_explicit}
X_t=\frac{x_0\exp\left(\eta t+\sigmab W_t\right)}
        {\sqrt{1+2 x_0^2 \lambda\int_0^t\exp\left(2\eta s+2 \sigmab W_s\right)ds}}
\end{equation}
for $ t \in [0,T] $.
Since
\begin{equation}\begin{split}
\left| 
\left( \eta + \frac{ \sigmab^2 }{2}
\right) x - \lambda x^3
\right|
&\geq
\lambda 
\left| x \right|^3
-
\left( \eta + \frac{ \sigmab^2 }{2}
\right)
\left| x \right|
\\
&\geq
\frac{\lambda}{2}
\left| x \right|^3
+  
\frac{\lambda}{2}
\left| x \right|^3
-
\left( \eta + \frac{ \sigmab^2 }{2}
\right)
\left| x \right|
\\
&=
\frac{\lambda}{2}
\left| x \right|^3
+
\frac{\lambda}{2}
\left| x \right|
\left( 
\left| x \right|^2
-
\frac{ 2 \eta + \sigmab^2 
}{ \lambda }
\right)
\\
&\geq
\frac{\lambda}{2}
\left| x \right|^3
+
\frac{\lambda}{2}
\left| x \right|
\left( 
\left| x \right|
-
\frac{ 2 \eta + \sigmab^2 
}{ \lambda }
\right)
\geq
\frac{\lambda}{2}
\left| x \right|^3
\end{split}
\end{equation}
for all 
$ \left| x \right|
\geq \max\left(1, 
\frac{ 2 \eta + \sigmab^2 
}{ \lambda }
\right) $,
the constants in 
condition~\eqref{eq:fdom} can be chosen as
$\beta=3$, $\alpha=1$,
$C=\max\left(1,\sigmab,
\frac{2}{\lambda},
\tfrac{2\eta+\sigmab^2}{\lambda}
\right)$.
% This choice of $C$ ensures
% $|\lambda x^3-(\eta+\tfrac{1}{2}\sigmab^2)|\geq\tfrac{\lambda}{2}|x|^3$
% for all $|x|\geq C$.

\subsection{Stochastic 
Verhulst equation}
The Verhulst equation is an ordinary differential equation and is a simple
model for a population with competition between individuals.
Its stochastic version with multiplicative noise can be written as
\begin{equation} \label{eq:verhulst}
d X_t =\left(
 \left( \eta + \frac{1}{2}\sigmab^2
 \right) 
 X_t 
 -
 \lambda X_t^2 
 \right) dt + \sigmab X_t \, dW_t,
\qquad
X_0 = x_0\in(0,\infty)
\end{equation}
for $ t \in [0,T] $
where
$\eta, \lambda,\sigmab>0$.
Its solution is known explicitly (see, e.g., Section 4.4 in
Kloeden \& 
Platen~1992) 
and is given by
\begin{equation}  \label{eq:verhulst_explicit}
X_t=\frac{x_0\exp\left(\eta t+\sigmab W_t\right)}
        {1+x_0 \lambda\int_0^t\exp\left(\eta s+\sigmab W_s\right)ds}
\end{equation}
for $ t \in [0,T] $.
% To meet the requirements of Theorem~\ref{thm:euler} we extend the drift function
% to the negative reals as $(\eta+\tfrac{1}{2}\sigmab^2)x-\operatorname{sgn}(x)x^2$.
The dominant exponent of the drift function is two and of the diffusion function
is one.
Thus we may choose $\beta=2$ and $\al=1$.
The constant $C$ in 
condition~\eqref{eq:fdom} can be chosen as
$C=\max\left(\sigmab,
\tfrac{2}{\lambda},\tfrac{2\eta+\sigmab^2}{\lambda}\right)$.

\subsection{Feller diffusion 
with logistic growth}
The branching process with logistic growth (see, e.g., 
Lambert~2005) 
is a stochastic Verhulst 
equation with Feller noise. It solves
\begin{equation}  \label{eq:fellog}
d X_t =\lambda X_t
 \left(K- X_t \right) dt 
+ \sigmab \sqrt{ X_t}\, dW_t,
\qquad
X_0 = x_0\in(0,\infty)
\end{equation}
for $ t \in [0,T] $
where $\lambda, K,\sigmab>0$. 
There is no explicit solution for 
this equation.
However, 
it features the following self-duality
\begin{equation}
\mathbb{E}^x \exp\left(-\frac{2\lambda}{\sigmab^2}X_t y\right)=
\mathbb{E}^y \exp\left(-\frac{2\lambda}{\sigmab^2}x X_t \right)
\end{equation}
for all $ t \in [0,T] $
and all $ x,y \in [0,\infty) $
where $\mathbb{E}^x$ refers to the starting point $X_0=x$, see 
Hutzenthaler \& 
Wakolbinger (2007).
As constants for 
condition~\eqref{eq:fdom} 
serve $\beta=2$, $\al=1$,
$C=\max(1,\tfrac{2}{\lambda},\sigmab,2K)$.

\subsection{Protein 
Kinetics}
The  proportion $x$ of one form of a certain protein can be modelled
by an ordinary differential equation whose appropriate stochastic version 
is given by
\begin{multline*}
  d X_t = \left( \eta - X_t + 
   \lambda X_t \left( 1 - X_t \right)
             +\frac{1}{2}\sigmab^2 
               X_t \left( 1 - X_t
               \right) \left( 1 - 2 X_t
               \right)
             \right) \! dt \\
             + \sigmab X_t 
             \left( 1 - X_t \right) 
             dW_t
\end{multline*}
for $ t \in [0,T] $
where $X_0=x_0\in(0,1)$, $\eta\in[0,1]$ and 
$\lambda, \sigmab>0$, see Section 7.1 in Kloeden
\& Platen~(1992).
Here Theorem~\ref{thm:euler} applies with $\beta=3$, $\al=2$ and $C$ sufficiently large.

\section{Simulations}\label{sec:simulations}
In this section we present two 
numerical simulations which
illustrate the divergence of Euler's
method as formulated in 
Theorem~\ref{thm:euler}.
For this purpose we choose an 
equation with an explicit 
solution
to compare with. Consider the stochastic differential equation
\begin{equation}  \label{eq:ginzburg-landau_simulation}
dX_t=- \left( \frac{1}{2} \sigmab^2 X_t - X_t^3 \right) dt+\sigmab X_t \, d W_t,\qquad X_0=1
\end{equation}
for $ t \in [0,3] $
where $ \sigmab > 0 $
is a constant.
This equation agrees with the stochastic Ginzburg-Landau equation~\eqref{eq:ginzburg-landau}
in the case $\eta=0,\lambda=1,x_0=1$.
Its explicit solution 
is given by equation~\eqref{eq:ginzburg-landau_explicit}
with $\eta=0,\lambda=1,x_0=1$.

Table~\ref{tab1} shows Monte Carlo simulations of the second 
moment both of
the exact solution
$ \mathbb{E}\!\left[ 
( X_3 )^2 \right]$ 
and of the Euler approximation 
$\mathbb{E}\!\left[ ( Y^N_N )^2
\right]$
for five values of $\sigmab $.
For each value of $\sigmab$ the table contains ten simulation runs
for the Euler approximation. The number
of time steps is $N=10^3$.
\begin{table}
\begin{center}
\begin{tabular}{|c||c|c|c|c|c|} \hline
&$\sigmab=2$
&$\sigmab=4$
&$\sigmab=5$
&$\sigmab=6$
&$\sigmab=7$
\\ \hline\hline
Exact solution 
$\mathbb{E}\!\left[ ( X_3 )^2
\right]$
&$0.4739$
% 0.474020, 0.474269, 0.475626
% 0.473329, 0.474265, 0.473014
% 0.473812, 0.474311, 0.474289
% 0.472294
&$0.9598$
% 0.955779, 0.959367, 0.957528
% 0.963198, 0.956439, 0.953888
% 0.959433, 0.965902, 0.964101
% 0.962235
&$1.2357$
% 1.228932, 1.238750, 1.234725,
% 1.242253, 1.237152, 1.234627,
% 1.233486, 1.234265, 1.245110,
% 1.228200
&$1.5423$ 
%1.5364, 1.5386, 1.5473, 1.5442, 1.5512,
%1.5269, 1.5396, 1.5392, 1.5502, 1.5493
&$1.8861$
% 1.879450, 1.880267, 1.902922
% 1.867702, 1.903806, 1.885652
% 1.883130, 1.888679, 1.884919
% 1.884065
\\ \hline
\hline
Simulation $1$
& $0.4577$ 
& $0.7448$ 
& $0.7243$ 
& NaN 
& NaN 
\\
\hline
Simulation $2$
& $0.4549$ 
& $0.7518$ 
& $0.7094$ 
& $0.5097$ 
& NaN 
\\
\hline
Simulation $3$
& $0.4632$ 
& $0.7330$ 
& $0.6912$ 
& NaN 
& NaN 
\\
\hline
Simulation $4$
& $0.4580$ 
& $0.7544$ 
& $0.7150$ 
& $0.5378$ 
& NaN 
\\
\hline
Simulation $5$
& $0.4619$ 
& $0.7458$ 
& $0.6998$ 
& $0.5197$ 
& NaN 
\\
\hline
Simulation $6$
& $0.4610$ 
& $0.7664$ 
& $0.7327$ 
& $0.5243$ 
& NaN 
\\
\hline
Simulation $7$
& $0.4606$ 
& $0.7542$ 
& $0.7189$ 
& NaN 
& NaN 
\\
\hline
Simulation $8$
& $0.4690$ 
& $0.7451$ 
& $0.7122$ 
& NaN 
& NaN 
\\
\hline
Simulation $9$
& $0.4667$ 
& $0.7336$ 
& $0.7146$ 
& $0.5475$ 
& NaN 
\\
\hline
Simulation $10$
& $0.4591$ 
& $0.7605$ 
& $0.6659$ 
& NaN 
& NaN 
\\
\hline
\end{tabular}
\caption{\footnotesize Simulation 
of $ \mathbb{E}\!\left[
( X_3 )^2
\right] $
and
$
\mathbb{E}\!\left[
( Y^N_N )^2
\right]
$
for the 
SDE~\eqref{eq:ginzburg-landau_simulation} 
where the number
of time steps is $N=10^3$. The number of Monte Carlo runs for the 
Euler approximation is
$10^5$ and for the exact 
solution is $10^7$.}
\label{tab1}
\end{center}
\end{table}
In the case $\sigmab=7$ the parameters of the simulation
are such that the Euler approximation produces the
value ``NaN'' 
(NaN is the IEEE arithmetic 
representation 
for ``not-a-number'',
here this is due to an operation 
``Inf $-$ Inf'' where Inf is 
the IEEE arithmetic
representation for positive infinity).
In contrast to this
our Monte Carlo simulations in the cases $\sigmab\in\{2,4,5\}$
are all finite.
In these cases the probability of the event on which the Euler
approximation diverges seems to be rather small.
The last but one column in Table~\ref{tab1} exemplifies a parameter setting in which the
event of large growth has a probability such that in $10^5$ runs in some Monte Carlo
simulations no explosion occurs and in some Monte Carlo simulations at least
one explosion occurs.

The values in Table~\ref{tab1} are either within distance two of the
true value or ``NaN''. 
This is due to
the double-exponential growth of the deterministic system for some initial values.
If the simulation starts to grow, then it reaches ``NaN'' very quickly.
We encountered similar behaviour for other exponents greater than two.
In order to see double-exponential growth of the Euler approximation in a plot
we consider an exponent close to one. More precisely, 
we plot the
Monte Carlo simulation of 
the first absolute moment 
of the Euler approximation
of $X_{10}$ where
\begin{figure}
\includegraphics[width=\textwidth]{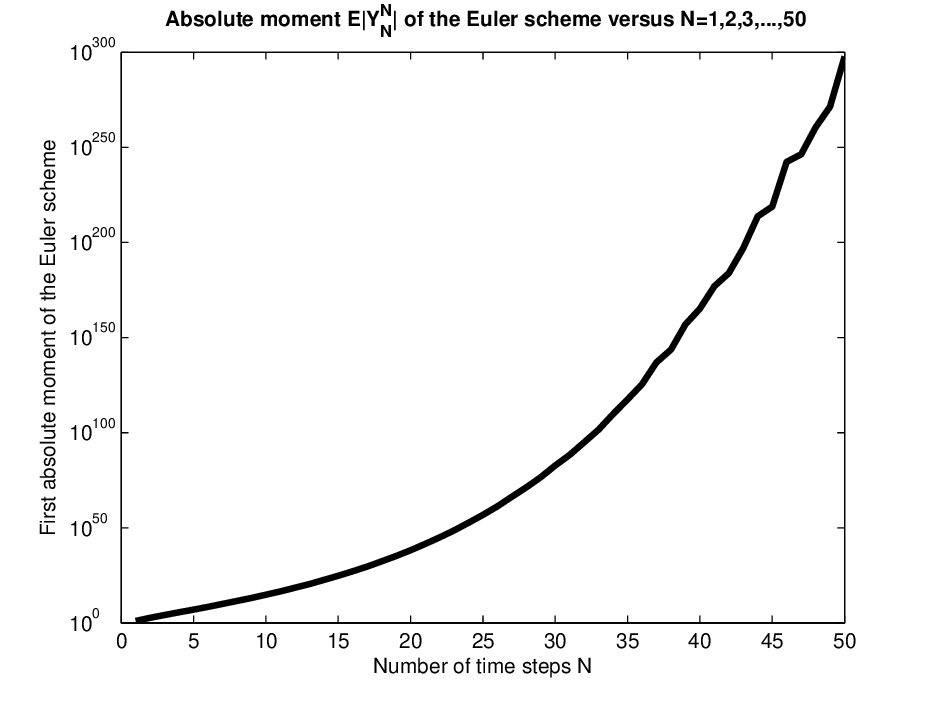}
\caption{\footnotesize
\label{f:plot}
Simulation of the
first absolute moment
$
  \mathbb{E}\!\left[ 
    | Y^N_N |
  \right] 
$
of the Euler scheme
for the SDE~\eqref{eq:complicated}
and
for 
number of time steps
$N \in \left\{ 1,2,3,\dots,50 \right\} $. 
The number of Monte Carlo 
runs is $10^4$. }
\end{figure}
\begin{equation}\label{eq:complicated}
d X_t=
-10\,\mathrm{sgn}(X_t) 
\left| X_t \right|^{1.1} dt
+ 4\,dW_t,
\qquad
X_0 = 0
\end{equation}
for $ t \in [0,10] $
see Figure~\ref{f:plot}
(see also the simulation values in Table~\ref{tab2}).
Note that the graph resembles an
exponential function and 
that the $y$-axis is logarithmic.
Thus the growth of the graph in Figure~\ref{f:plot} is indeed
close to a double-exponential function.
\newlength{\halfpage}
\setlength{\halfpage}{\textwidth/2-1mm}
\begin{table}
\noindent
\begin{minipage}{\halfpage}
\begin{center}
\begin{tabular}{|c||c|} \hline
$N$ &
$
  \mathbb{E}\!\left[ 
    | Y^N_N |
  \right] 
$
\\ \hline
\hline
$N=1$ & $10.1160$ \\ \hline
$N=2$ & $444.5797$ \\ \hline
$N=3$ & $1.3649 \cdot 10^4$ \\ \hline
$N=4$ & $3.7566 \cdot 10^5$ \\ \hline
$N=5$ & $1.0295 \cdot 10^7$ \\ \hline
\vdots& \vdots \\ \hline
\end{tabular}
\end{center}
\end{minipage}
\begin{minipage}{\halfpage}
\begin{center}
\begin{tabular}{|c||c|} \hline
$N$ &
$
  \mathbb{E}\!\left[ 
    | Y^N_N |
  \right] 
$
\\ \hline
\hline
$N=48$ & $3.4449 \cdot 10^{260}$ \\ \hline
$N=49$ & $2.6631 \cdot 10^{271}$ \\ \hline
$N=50$ & $6.1140 \cdot 10^{297}$ \\ \hline
$N=51$ & $7.2342 \cdot 10^{301}$ \\ \hline
$N=52$ & Inf \\ \hline
$N=53$ & Inf \\ \hline
\end{tabular}
\end{center}
\end{minipage}
\caption{\footnotesize Simulation values of the first absolute moment
$
  \mathbb{E}\!\left[ 
    | Y^N_N |
  \right] 
$
for $N \in \left\{ 1,2,3,\dots,53 
\right\} $ and
the SDE~\eqref{eq:complicated}.}
\label{tab2}
\end{table}
In addition we should mention that some fine-tuning was necessary
to obtain suitable parameters for which the simulated absolute moment
grows but is not ``NaN''.

\section{Proofs}\label{sec:proofs}

\begin{lemma}\label{normalrv}
Let $ \left( \Omega, \mathcal{F},
\mathbb{P} \right) $ be
a probability space and let
$ Z \colon \Omega \rightarrow \mathbb{R} $
be an $ \mathcal{F} $/$ \mathcal{B}(
\mathbb{R})$-measurable mapping
which is standard normal distributed. 
Then
\begin{equation}
\mathbb{P}\Big[ 
  | Z | \geq x
\Big]
\geq 
\frac{ 
 x \, e^{ - x^2 }
}{4},
\qquad
\mathbb{P}\Big[ 
 | Z | 
 \in [ x, 2 x ]
\Big]
\geq 
\frac{ x \, e^{ - 2 x^2 } }{2}
\end{equation}
for all $ x \in [0,\infty)$.
\end{lemma}
\begin{proof}[Proof of Lemma 
\ref{normalrv}]
We have
\begin{equation}  \begin{split}
\mathbb{P}\Big[ 
  | Z | \geq x 
\Big]
&=
2 
\cdot
\mathbb{P}\Big[ 
  Z \geq x
\Big]
= 
  2
  \left(
    \int^{ \infty }_{ x }
      \frac{1}{ \sqrt{ 2 \pi } }
      e^{ - \frac{s^2}{2} }
    ds
  \right)
\\&\geq
  2
  \left(
    \int^{ \sqrt{2} x }_{ x }
      \frac{1}{ \sqrt{ 2 \pi } }
      e^{ - \frac{s^2}{2} }
    ds
  \right)
\\&\geq
  2 \left( \sqrt{2} - 1 \right) x
  \left(
    \frac{1}{ \sqrt{ 2 \pi } }
    e^{ 
      - \frac{ \left( \sqrt{2} x \right)^2 
      }{2} 
    }
  \right)
\\&=
  \left(
    \frac{ 2 \left( \sqrt{2} - 1 \right) }{ 
      \sqrt{ 2 \pi } }
  \right)
  x e^{ - x^2 }
\geq
  \frac{ x e^{ - x^2 } }{4}
\end{split}     \end{equation}
for all $ x \in [0,\infty) $.
Moreover, a similar estimate yields
\begin{equation}  \begin{split}
\mathbb{P}\Big[ 
  | Z | \in [x,2x]
\Big]
&=
2 
\cdot
\mathbb{P}\Big[ 
  Z \in [x,2x]
\Big]
= 
2
\left(
\int^{ 2 x }_{ 
  x
}
\frac{1}{ \sqrt{ 2 \pi } }
e^{ - \frac{s^2}{2} }
ds
\right)
\\&\geq
 2 x
 \left(
 \frac{1}{ \sqrt{ 2 \pi } }
 e^{ - \frac{ \left( 2 x \right)^2 
 }{2} }
 \right)
=
 \left(
   \frac{ 2 }{ 
     \sqrt{ 2 \pi } }
 \right)
 x e^{ - 2 x^2 }
\geq
\frac{ x \, e^{ - 2 x^2 } }{2}
\end{split}     \end{equation}
for all $ x \in [0,\infty)$.
\end{proof}
\begin{proof}[Proof of 
Theorem \ref{thm:euler}]
By assumption we have
$\mathbb{P}[\,| \sigma(\xi) | > 0\,]>0$.
Therefore, there exists a
real number $ K \in (1,\infty) $, such that
\begin{equation}  \label{eq:def_mu}
\vartheta:=\mathbb{P}\!\left[ \,
 \left| \sigma(\xi) \right| 
 \geq \frac{1}{K},
 \;\;
 \left| \xi \right|
 + T \left| \mu(\xi) \right|
 \leq K
 \,
\right] > 0.
\end{equation}
Then we define
\begin{equation}\label{rn}
r_N := 
\max\!\left( 2, C,
\left( 
 \frac{ 2 C N }{ T } 
 + 2 C^2
\right)^{ 
  \!\!
  \frac{1}{
     \left( \beta - \alpha \right)
  } 
} 
\right)
\in [C,\infty)
\end{equation} and
consider the sets
$ \Omega_N \in \mathcal{F} $
given by
\begin{multline}
\label{omegan}
\Omega_N :=
\bigg\{
 \omega \in \Omega 
 \,\bigg|\,
\left|
W_{ \frac{ (n+1)T}{N} }(\omega)
-
W_{ \frac{ n T}{N} }(\omega)
\right|
\in
\left[ 
 \frac{T}{N},
 \frac{2 T}{N}
\right]
\forall 
\,
n \in \left\{ 1, \dots, N-1 \right\} ,
\\
 \left|
   W_{ \frac{ T }{ N } }(\omega)
   - W_0(\omega)
 \right|
 \geq K \left( r_N + K \right),
\\
 \left| \sigma(\xi(\omega)) 
 \right|
 \geq \frac{1}{K},
\;\;
   \left| \xi(\omega) \right|
   +
   T \left| \mu(\xi(\omega)) 
   \right|
 \leq K    
\bigg\}
\end{multline}
for all $ N \in \mathbb{N}$.
Note that
\begin{equation}\label{thisrN}
\frac{T}{N C}
 \left( r_N 
 \right)^{\left(\beta-\alpha\right)}
\geq 
\frac{T}{N C}
 \left( 
   \frac{ 2 C N }{ T } 
   + 2 C^2
 \right)
\geq
2
+ 
\frac{2 T C}{N}
\end{equation}
for all $ N \in \mathbb{N}$
due to the definition of $r_N$,
$ N \in \mathbb{N}$.

Let now $ N \in \mathbb{N}$ 
and $ \omega \in \Omega_N $ 
be arbitrary.
Then we claim
\begin{equation}\label{toshowhier}
\left| Y^N_n(\omega)
\right|
\geq
\left( r_N \right)^{
 \left(
   \alpha^{\left( n - 1 \right) }
 \right)
}
\end{equation}
for all 
$ n \in 
\left\{ 1,2,\dots,N 
\right\} $.
We prove equation \eqref{toshowhier}
by induction 
on $n
\in \left\{ 1, 2, \dots, N 
\right\}$.
In the base case $n=1$ we have
\begin{equation}  \begin{split}
\left|
 Y^N_1(\omega)
\right|
&=
\left|
 \xi(\omega)
 +
 \frac{T}{N}
 \mu\left(
   \xi(\omega)
 \right)
 +
 \sigma\left(
   \xi(\omega)
 \right)\left(
   W_{ \frac{T}{N} }(\omega)
   -
   W_{ 0 }(\omega)
 \right)
\right|
\\&\geq
\left|
 \sigma\left(
   \xi(\omega)
 \right)
\right| \cdot
\left|
   W_{ \frac{T}{N} }(\omega)
   -
   W_{ 0 }(\omega)
\right|
-
\left|
 \xi(\omega)
\right|
 -
 \frac{T}{N}
\left|
 \mu\left(
   \xi(\omega)
 \right)
\right|
\\&\geq
\left|
 \sigma\left(
   \xi(\omega)
 \right)
\right| \cdot
\left|
   W_{ \frac{T}{N} }(\omega)
   -
   W_{ 0 }(\omega)
\right|
-
K
\\&\geq
\frac{1}{K}
\left|
   W_{ \frac{T}{N} }(\omega)
   -
   W_{ 0 }(\omega)
\right|
-
K
\geq
\frac{
 K \left( r_N + K \right)
}{K}
- K
= r_N
\end{split}     \end{equation}
due to the definition~\eqref{corexpeuler}
of $Y_1^N$
and due to the definition~\eqref{omegan}
of $\Omega_N$.
For the induction step $ n \to n+1 $ 
we assume
that equation \eqref{toshowhier}
holds for one
$ n \in \left\{ 1, 2, \dots, N - 1 \right\}$.
In particular, we then obtain
\begin{equation}\label{yrN}
\left|
 Y^N_n(\omega) 
\right|
\geq
r_N
\geq C
\geq 1
\end{equation}
since $ r_N \geq 2 $.
Additionally, we have
\begin{align*}
\left| 
 Y^N_{n+1}(\omega)
\right|
&=
\left|
 Y^N_{n}(\omega)
 + \frac{T}{N} 
   \mu( Y^N_{n+1}(\omega) )
 +
 \sigma( Y^N_{n+1}(\omega) ) 
 \left(
   W_{ \frac{ (n+1) T }{ N } }(\omega)
   -
   W_{ \frac{ n T }{ N } }(\omega)
 \right)
\right|
\\&\geq
\left|
 \frac{T}{N} \mu(
   Y^N_{n}(\omega)
 )
 +
 \sigma(
   Y^N_{n}(\omega)
 ) 
 \left(
   W_{ \frac{ (n+1) T }{ N } }(\omega)
   -
   W_{ \frac{ n T }{ N } }(\omega)
 \right)
\right|
-
\left| 
 Y^N_{n}(\omega)
\right|
\\&\geq
\max\left( 
 \frac{T}{N} 
 \left|
   \mu(
     Y^N_{n}(\omega)
   )
 \right|,
 \left|
 \sigma(
   Y^N_{n}(\omega)
 )
 \right| \cdot
 \left|
   W_{ \frac{ (n+1) T }{ N } }(\omega)
   -
   W_{ \frac{ n T }{ N } }(\omega)
 \right|
\right)
\\&\;\;
- \min\left( 
 \frac{T}{N} 
 \left|
   \mu(
     Y^N_{n}(\omega)
   )
 \right|,
 \left|
 \sigma(
   Y^N_{n}(\omega)
 )
 \right| \cdot
 \left|
   W_{ \frac{ (n+1) T }{ N } }(\omega)
   -
   W_{ \frac{ n T }{ N } }(\omega)
 \right|
\right)
\\&\;\;
-
\left| 
 Y^N_{n}(\omega)
\right|
\end{align*}
and
\begin{equation}  \begin{split}
\left| 
 Y^N_{n+1}(\omega)
\right|
&\geq
\max\left( 
 \frac{T}{N} 
 \left|
   \mu(
     Y^N_{n}(\omega)
   )
 \right|,
 \frac{T}{N}
 \left|
 \sigma(
   Y^N_{n}(\omega)
 )
 \right|
\right)
\\&\;\;
- \min\left( 
 \frac{T}{N} 
 \left|
   \mu(
     Y^N_{n}(\omega)
   )
 \right|,
 \frac{2 T}{N}
 \left|
 \sigma(
   Y^N_{n}(\omega)
 )
 \right| 
\right)
-
\left| 
 Y^N_{n}(\omega)
\right|
\\&\geq
\frac{T}{N}
\max\left( 
 \left|
   \mu(
     Y^N_{n}(\omega)
   )
 \right|,
 \left|
 \sigma(
   Y^N_{n}(\omega)
 )
 \right|
\right)
\\&\;\;
- 
\frac{2 T}{N}
\min\left( 
 \left|
   \mu(
     Y^N_{n}(\omega)
   )
 \right|,
 \left|
 \sigma(
   Y^N_{n}(\omega)
 )
 \right| 
\right)
-
\left| 
 Y^N_{n}(\omega)
\right|
\end{split}     \end{equation}
due to definition~\eqref{corexpeuler}
and definition~\eqref{omegan}.
Therefore,
the growth 
condition~\eqref{eq:maxmin_condition}
and
inequality \eqref{yrN}
imply
\begin{equation}  \begin{split}
\left| 
 Y^N_{n+1}(\omega)
\right|
&\geq
\frac{T}{N C}
 \left|
     Y^N_{n}(\omega)
 \right|^{\beta}
- 
\frac{2 T C}{N}
 \left|
     Y^N_{n}(\omega)
 \right|^{\alpha}
-
\left| 
 Y^N_{n}(\omega)
\right|
\\&\geq
\frac{T}{N C}
 \left|
     Y^N_{n}(\omega)
 \right|^{\beta}
- 
\frac{2 T C}{N}
 \left|
     Y^N_{n}(\omega)
 \right|^{\alpha}
-
\left| 
 Y^N_{n}(\omega)
\right|^{\alpha}
\\&=
\left| 
 Y^N_{n}(\omega)
\right|^{\alpha}
\left(
\frac{T}{N C}
 \left|
     Y^N_{n}(\omega)
 \right|^{\left(\beta-\alpha\right)}
- 
\frac{2 T C}{N}
- 1
\right)
\\&\geq
\left| 
 Y^N_{n}(\omega)
\right|^{\alpha}
\left(
\frac{T}{N C}
 \left( r_N 
 \right)^{\left(\beta-\alpha\right)}
- 
\frac{2 T C}{N}
- 1
\right)
\geq
\left| 
 Y^N_{n}(\omega)
\right|^{\alpha}
\end{split}     \end{equation}
where the last step
follows from
inequality \eqref{thisrN}.
The induction
hypothesis hence yields
\begin{equation}
\left| 
 Y^N_{n+1}(\omega)
\right|
\geq
\left| 
 Y^N_{n}(\omega)
\right|^{\alpha}
\geq
\left(
 \left( r_N \right)^{
   \left(
   \alpha^{\left(n-1\right)}
 \right) }
\right)^{ \alpha }
=
 \left( r_N \right)^{
   \left(
   \alpha^n
   \right)
 } .
\end{equation}
This proves
inequality~\eqref{toshowhier}
for $n+1$.
Therefore, equation 
\eqref{toshowhier} indeed holds
for all $ n \in \left\{ 1,2,\dots,N
\right\}$.
In particular
-- since $ N \in \mathbb{N} $
and $ \omega \in \Omega_N $ were 
arbitrary --
we obtain
\begin{equation}
\label{eq:doubleexp}
  \left| Y^N_N(\omega)
  \right|
  \geq
  \left(
    r_N
  \right)^{ 
    \left( 
      \alpha^{ \left( N-1 \right) } 
    \right) 
  }
  \geq
    2 ^{ 
    \left( 
      \alpha^{ \left( N-1 \right) } 
    \right) 
  }
\end{equation}
for all 
$ \omega \in \Omega_N$
and all 
$ N \in \mathbb{N}$.

Furthermore, 
definition~\eqref{eq:def_mu} 
gives
\begin{align*}
  \mathbb{P}\big[
    \Omega_N
  \big]
&=
  \vartheta  
  \cdot 
  \mathbb{P}\bigg[
    \left|
      W_{ \frac{ T}{N} }
    \right| 
    \geq 
    K ( r_N + K )
  \bigg]	
  \cdot
  \Bigg(
    \mathbb{P}\!\left[ \,
      \left|
        W_{ \frac{ T}{N} }
      \right|
      \in
      \left[ 
        \frac{T}{N},
        \frac{2 T}{N}
      \right]
    \right]
  \Bigg)^{ \!\! ( N - 1 ) }
\\
  &\geq
  \vartheta  
  \cdot 
  \mathbb{P}\Big[
    \left|
      W_1
    \right| 
    \geq 
    N^{ \frac{1}{2} } 
    T^{ - \frac{1}{2} } 
    K 
    ( r_N + K )
  \Big]	
  \cdot
  \Bigg(
    \mathbb{P}\!\left[ \,
      \left|
        W_1
      \right|
      \in
      \left[ 
        \frac{ \sqrt{T} }{ \sqrt{N} },
        \frac{ 2 \sqrt{T} }{ \sqrt{N} }
      \right]
    \right]
  \Bigg)^{ \!\! N }
\end{align*}
for all $ N \in \mathbb{N} $.
Lemma~\ref{normalrv} therefore yields
\begin{align}
  \mathbb{P}\big[
    \Omega_N
  \big]
&\geq
  \frac{
    \vartheta  
    \sqrt{N} 
    K 
    ( r_N + K )
  }{
    4 \sqrt{T}
  }
  \cdot
  \exp\!\left( 
    - 
    \frac{ 
      N K^2  
      ( r_N + K )^2
    }{ 
      T 
    }
    - 2 T
  \right)
  \cdot
  \Bigg(
    \frac{
      \sqrt{T}
    }{
      2 \sqrt{N}
    }
  \Bigg)^{ \!\! N }  
\nonumber
\\&\geq
  \frac{
    \vartheta  
    e^{ - 2 T }
    T^{ \frac{ (N-1) }{ 2 } }
  }{
    2^{ (N+2) } 
    N^N
  }
  \cdot
  \exp\!\Big( 
    - 
    N T^{-1} K^2  
    ( r_N + K )^2
  \Big)
\end{align}
for all $ N \in \mathbb{N} $.
This shows the existence of a constant
$ c \in (1,\infty) $ such that
\begin{equation}
\label{eq:omegaN}
  \mathbb{P}\big[ \Omega_N 
  \big] 
  \geq 
  c^{ \left( - N^c \right) } 
\end{equation}
for all $ N \in \mathbb{N} $.

Combining \eqref{eq:doubleexp}
and \eqref{eq:omegaN} then gives
\begin{multline}
  \lim_{ N \rightarrow \infty }
  \mathbb{E}\Big[
    \left|  
      Y^N_N
    \right|  
  \Big]
\geq
  \lim_{ N \rightarrow \infty }
  \mathbb{E}\Big[
    1_{ \Omega_N }
    \!
    \left|  
      Y^N_N
    \right|
  \Big]
\geq
  \lim_{ N \rightarrow \infty }
  \left(
    \mathbb{P}\big[ \Omega_N \big]
    \cdot	
    \left(
      r_N
    \right)^{
      \left(
        \alpha^{
          \left(N-1\right)
        } 
      \right) 
    }
  \right)
\\ \geq
  \lim_{ N \rightarrow \infty }
  \left(
    \mathbb{P}\big[ \Omega_N \big]
    \cdot	
    2^{
      \left(
        \alpha^{
          \left( N - 1 \right)
        } 
      \right) 
    }
  \right)
\geq
  \lim_{ N \rightarrow \infty }
  \left(
    c^{ \left( - N^c \right) }
    \cdot	
    2^{
      \left(
        \alpha^{
          \left(N-1\right)
        } 
      \right) 
    }
  \right)
  = \infty
\end{multline}
and Jensen's inequality hence shows
\begin{equation}
\lim_{N \rightarrow \infty}
\mathbb{E}\Big[\left| 
 Y^N_N
\right|^p \Big]= \infty .
\end{equation}
Since $ 
\mathbb{E}\big[ | X_T |^p \big]
< \infty $ by assumption,
we finally conclude
\begin{equation}
  \lim_{N \rightarrow \infty}
  \mathbb{E}\Big[
    \left|
      X_T - Y^N_N
    \right|^p
  \Big] 
  = \infty
  \quad
  \text{and}
  \quad
  \lim_{N \rightarrow \infty}
  \bigg| 
    \mathbb{E}\Big[
      \left|
        X_T 
      \right|^p
    \Big]
    - 
    \mathbb{E}\Big[
      \left| Y^N_N
      \right|^p 
    \Big]
  \bigg| 
  = \infty .
\end{equation}
This, \eqref{eq:doubleexp} 
and \eqref{eq:omegaN} 
finally complete the proof 
of Theorem~\ref{thm:euler}.
\end{proof}

\begin{acknowledgements}
This work has been partially 
supported by the 
Collaborative Research Centre 701
"Spectral Structures and Topological Methods in Mathematics" funded by the German Research
Foundation. Moreover, the authors thank three anonymous referees for their helpful comments.
\end{acknowledgements}
%\bibliographystyle{acm}
%\bibliographystyle{abbrv}
%\bibliography{./bibliography_euler}

\label{lastpage}
\end{document}